\documentclass[reqno,a4paper]{amsart}
\usepackage[utf8]{inputenc}
\usepackage[T1]{fontenc}
\usepackage[english]{babel}
\usepackage{amssymb}
\usepackage[foot]{amsaddr}

\usepackage{mathrsfs}

\usepackage{graphicx}
\usepackage{wrapfig}
\usepackage{subfig}
\usepackage{enumerate}
\usepackage[usenames,dvipsnames]{color}
\usepackage{dsfont}
\usepackage{microtype}

\usepackage{nomencl}

\usepackage{etoolbox}

\usepackage{mathtools}
\mathtoolsset{showonlyrefs,showmanualtags}

\usepackage{pstricks,pst-plot,pst-math} 
\usepackage{pstricks-add} 

\newtheorem{thm}{Theorem}[section]

\theoremstyle{definition}

\theoremstyle{remark}

\numberwithin{equation}{section}

\newtheorem{Theorem}[thm]{Theorem}
\newtheorem{Lemma}[thm]{Lemma}

\newcommand{\DeclareAutoPairedDelimiter}[3]{%
	\expandafter\DeclarePairedDelimiter\csname Auto\string#1\endcsname{#2}{#3}%
	\begingroup\edef\x{\endgroup
		\noexpand\DeclareRobustCommand{\noexpand#1}{%
			\expandafter\noexpand\csname Auto\string#1\endcsname*}}%
	\x}

\DeclareAutoPairedDelimiter{\abs}{\lvert}{\rvert}
\DeclareAutoPairedDelimiter{\norm}{\lVert}{\rVert}
\DeclareAutoPairedDelimiter{\bra}{(}{ )}
\DeclareAutoPairedDelimiter{\pra}{[}{]}
\DeclareAutoPairedDelimiter{\set}{\{}{\}}
\DeclareAutoPairedDelimiter{\skp}{\langle}{\rangle}

\DeclareMathAlphabet{\mathup}{OT1}{\familydefault}{m}{n}
\newcommand{\dx}[1]{\mathop{}\!\mathup{d} #1}

\newcommand{\N}{\mathds{N}}
\newcommand{\R}{\mathds{R}}

\newcommand{\cL}{\ensuremath{\mathcal L}}

\newcommand{\sL}{\ensuremath{\mathscr L}}

\newcommand{\tw}{\ensuremath{\tilde w}}

\usepackage[colorinlistoftodos,prependcaption,textsize=tiny]{todonotes}

\newtoggle{Nebenrechnungen}
\toggletrue{Nebenrechnungen}

\usepackage{hyperref}
\definecolor{darkblue}{rgb}{0,0,0.6}
\hypersetup{
    pdftitle={A heteroclinic orbit connecting traveling waves pertaining to different nonlinearities},    
    pdfauthor={Simon Eberle},
    colorlinks=true,       
    linkcolor=darkblue,          
    citecolor=darkblue,        
    filecolor=darkblue,      
    urlcolor=darkblue           
}

\title[A heteroclinic orbit connecting traveling waves \dots]{A heteroclinic orbit connecting traveling waves pertaining to different nonlinearities}

\author{Simon Eberle$^1$}
\address{$^1$Fakultät für Mathematik, Universität Duisburg-Essen.}

\email{simon.eberle@uni-due.de}

\let\rho\varrho
\let\epsilon\varepsilon

\begin{document}

\begin{abstract}
\noindent
In this paper we consider a semilinear parabolic equation in an infinite cylinder. The spatially varying nonlinearity is such that  it connects two (spatially independent) bistable nonlinearities in a compact set in space.
We prove that, given such a setting, a traveling wave obeying the equation with the one bistable nonlinearity and starting at the respective side of the cylinder, will converge to a traveling wave solution prescribed by the nonlinearity on the other side.
	
\end{abstract}
	\maketitle

\section{Introduction}

Since the pioneering works of Kolmogorov et al. \cite{kolmogorov1937etude} the study of traveling wave solutions for semilinear parabolic equations of several types (the most prominent are bistable, ignition and KPP-type nonlinearities) has been an active field. In the case of bistable nonlinearities, that we will concern ourselves with in this article, we want to refer to the celebrated paper by Fife and McLeod \cite{FifeMcLeod} for existence and uniqueness as well as stability in one spatial dimension. We also want to mention the detailed study of traveling waves in cylinders by Berestycki and Nirenberg \cite{BerestyckiNirenberg}. Later Berestycki, Hamel et al. have broadened the field by studying generalizations of traveling waves in domains or with coefficients / nonlinearities that do not allow for traveling wave solutions. In the case of periodic media, this has led to the notion of pulsating fronts \cite{BerestyckiHamelPeriodicExcitableMedia} and recently they have generalized it to the notion of transition fronts that do not require any special properties of the domain (apart from sufficiently smooth boundary and infinite geodesic diameter)  or of the coefficients \cite{BerestyckiHamelGeneralizedTransitionFronts}. 
In this respect we found \cite{MatanoObst} very inspiring where Matano, Berestycki and Hamel use super- and subsolutions as in \cite{MatanoStab} to construct an entire solution that starts as a traveling wave solution for $t \rightarrow - \infty$, passes the obstacle and converges - given the compact obstacle is sufficiently regular - to the same traveling wave solution as $t \rightarrow +\infty$.
Another recent and very interesting work on the construction of generalized transition fronts is \cite{Zlatos} where the author studies generalized transition fronts in cylinders subject to a space dependent nonlinearity that is bounded form above and below by spatially independent ignition-type nonlinearities.
We are trying to investigate a similar problem as is investigated in \cite{Zlatos}, but in our case the nonlinearities are of bistable type and do only vary in a compact transition zone. In contrast to \cite{Zlatos} we are interested in the existence of a heteroclinic connection between two traveling fronts, which is stronger than the very relaxed notion of a transition front (as it is given in \cite{BerestyckiHamelGeneralizedTransitionFronts}).

In this paper we will occupy ourselves with the construction of a transition front in a cylinder $D= \R \times \Omega$. But thanks to the compactness of the transition region, we can construct a heteroclinic orbit between two traveling wave solutions.
To be more precise the nonlinearity $f(x,u)$ shall be such that
\begin{align} \label{bounds_on_f}
\left \{ \begin{aligned}
\begin{split}
f_2(u) \leq f(x,u) \leq f_1(u) \text{ for } x \in D, u \in [0,1], \\
f(x,u ) = f_1(u) \text{ for } x_1 \geq 0, u \in [0,1] \text{ and } \\
f(x,u) = f_2(u) \text{ for } x_1 \leq -x_0, u \in [0,1],
\end{split}
\end{aligned} \right .
\end{align} 
where $f_1,f_2$ are two a-priori given nonlinearities of bistable type and $x_0 >0$ is a parameter of the transition region. 

\begin{figure}[!h]
	\psset{xunit=1cm,yunit=1cm}
	\frame{
		\begin{pspicture*}
		(-5.2,-1.2)(7.1,2.3)
		\psframe[fillstyle=vlines,hatchsep=0.2,hatchangle=120](2,2)
		\psline(-5,2)(7,2)
		\psline(-5,0)(7,0)
		\psline(0,0)(0,-0)
		\psline(2,0)(2,-0)
		\rput[lb](-1.5,0.8){$f_2$}
		\rput[lb](3.3,0.8){$f_1$}
		\rput[lb](-0.4,-0.8){$-x_0$}
		\rput[lb](1.9,-0.8){$0$}
		\end{pspicture*}
	}
	\caption{Infinite cylinder with transition zone}
\end{figure}
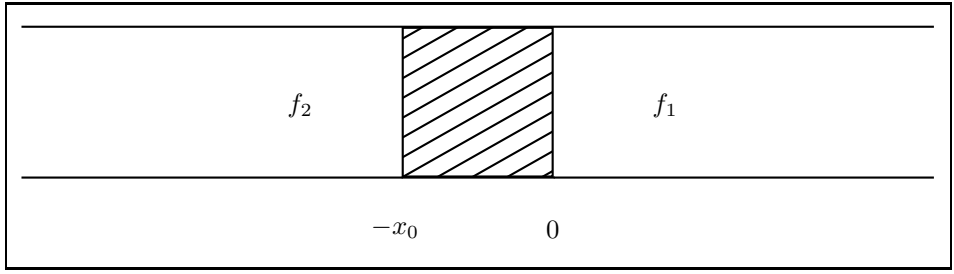

The main result of this article shall be

\begin{Theorem}
	Let $f$ satisfy \eqref{bounds_on_f}  (and \eqref{F_1}-\eqref{F_8}) then there is a unique entire solution $u(t,x)$ of  
	\begin{align} \label{DiffEqu}
	\begin{cases}
	\partial_t u - \Delta u = f(x,u) &\text{ in } D, \\
	\frac{\partial u}{\partial \nu} = 0 &\text{ on } \partial D,
	\end{cases}
	\end{align}
	
	such that $0 < u(t,x) <1$ and $\partial_t u(t,x) >0$ for all $(t,x) \in \R \times \bar{D}$ and
	\begin{align} 
	u(t,x)- \phi_1(x_1+c_1 t) \rightarrow 0 \text{ as } t \rightarrow -\infty \text{ uniformly in } x \in \bar{D},
	\end{align}
	then  
	\begin{align}
	u(t,x) - \phi_2(x_1 + c_2t + \beta) \rightarrow 0  \text{ as } t \rightarrow +\infty \text{ uniformly in } x \in \bar{D}
	\end{align}
	for some $\beta \in \R$.
\end{Theorem}

Here $(\phi_1,c_1)$ and $(\phi_2,c_2)$ are the one-dimensional traveling wave profiles and corresponding speeds solving 
\begin{align} \label{ODE_travelling_wave}
\begin{cases}
\phi_i^{\prime \prime}(z) - c_i \phi_i^\prime(z) + f_i(\phi_i(z)) =0, \\
\phi_i(- \infty) = 0, \phi_i(+ \infty)=1, \\
0<\phi_i(z) < 1\text{ for all } z \in \R .
\end{cases}
\end{align}

This means that we do not only prove that $u$ is a transition front, but also that the effect of the tail of the wave is negligible and it will in speed and profile be governed by $f_2$ for large times. 

We adapt the proof of existence and uniqueness in \cite{MatanoObst} to prove existence and uniqueness for our problem.
The proof of the long-term asymptotic behaviour, that is the crucial part in constructing this heteroclinic orbit, will mainly follow the structure of the argumentation in the celebrated paper by Fife and McLeod \cite{FifeMcLeod}. But it will also take in some inspiration gotten reading and re-reading the very helpful papers \cite{MatanoObst,MatanoStab}. The structure of our procedure will be roughly the following:

\begin{itemize}
	\item First of all we will estimate $u$ from below against some translation of $\phi_2(x_1+c_2t)$ and some correction terms that vanish for $t \rightarrow \infty$.
	\item Secondly we do the estimation of $u$ form above (note that this is the more complicated direction) against a traveling wave with profile $\phi_2$ and speed $c_2$ and again some error terms that vanish as $t \rightarrow +\infty$.
	\item In a third step we will employ the estimates from step 1 and 2 in a Lyapunov-function argument that will tell us that for a subsequence in time $t_n \rightarrow \infty$ the profile of the wave will converge to $\phi_2$ and the speed will converge to $c_2$ while translations cannot be ruled out, yet.
	\item Finally we prove uniqueness of the limit ($t \rightarrow \infty$) using a stability result that comes out of step 1 and 2 and controls the perturbation caused in the transition zone.
\end{itemize}

\section*{Acknowledgements}
We thank G. S. Weiss (supervisor) for the many fruitful discussions and François Hamel for pointing out a correction of the construction of entire solutions in \cite{MatanoObst} to the author.

\section{Notation and assumptions}
In this section we outline the assumptions we make and give the notation that shall be used in the following.
Let $D:= \R \times  \Omega$ be a cylindrical domain, i.e $\Omega \subset \R^{n-1}$ be an open, bounded set with smooth boundary.

The nonlinearity $f$ shall lie between two bistable nonlinearities $f_1$ and $f_2$ as in
\eqref{bounds_on_f}, that obey

\begin{align}
\label{F_1} &f_1, f_2 \in C^{1,1}([0,1]), \\
\label{F_2} &f_1(0) = f_2(0) = 0, f_1(1)= f_2(1) = 0, \\
\label{F_3}&f_1'(0) , f_2'(0) <0 , f_1'(1), f_2'(1) <0, \\
\label{F_4} 
& \left .\begin{aligned}
\begin{split}
 &f_1 <0 \text{ on } (0, \theta_1), f_1 >0 \text{ on } (\theta_1, 1), \\ 
&f_2 <0 \text{ on } (0, \theta_2), f_2 >0 \text{ on } (\theta_2, 1) , \\
&0 < \theta_1 < \theta_2 <1, 
\end{split}\end{aligned} \right \}
\\
\label{F_8}& \int \limits_0^1 f_1(u) \dx{u} , \int \limits_0^1 f_2(u) \dx{u} >0 .
\end{align}

Then there are unique speeds and unique (up to translation) traveling wave profiles for each of the nonlinearities $f_1$ and $f_2$, i.e
there are speeds $c_1, c_2 >0$ and wave profiles $\phi_1, \phi_2: \R \rightarrow \R$, such that (see e.g. \cite{FifeMcLeod}) 
\begin{align}
\text{ for } i \in \{1,2\}:
\begin{cases}
\phi_i(x_1+c_it)  \text{ solves } (\ref{DiffEqu}) \text{ with } f_i \text{ instead of } f \\
\phi_i(- \infty) = 0, \phi_i(+ \infty)=1 \\
0<\phi_i(z) < 1\text{ for all } z \in \R
\end{cases}
\end{align}
This is equivalent to $\phi_i$ solving the ordinary differential equation \eqref{ODE_travelling_wave}.

\section{construction of an entire solution which converges to a traveling wave $\phi_1(x_1+c_1t)$ for $t \rightarrow - \infty$}
Here we state for the sake of completeness the existence and uniqueness result as in \cite{MatanoObst} to show that there is an unique entire solution that admits the claimed asymptotic behaviour for $t \rightarrow - \infty$. The proof can be found in the appendix.

\begin{Theorem}[existence and uniqueness] \label{existence_uniqueness}
	Let $D$ and $f$ be as above. Then there is a \textit{unique} entire solution $u(t,x)$ of  (\ref{DiffEqu}) such that $0 < u(t,x)< 1$ and $\partial_t u(t,x)>0$ for all $(t, x) \in \R \times \bar{D}$ and 
	\begin{align} \label{Anfangsbedingung}
	u(t,x) - \phi_1(x_1+c_1t) \rightarrow 0 \quad \text{ as } t \rightarrow - \infty \text{ uniformly in } x \in \bar{D}.
	\end{align}
	
\end{Theorem}

\begin{proof}
	The proof mainly resembles the respective one in \cite{MatanoObst} and is therefore deferred to the appendix.
\end{proof}

\section{Long time convergence to a traveling front with profile $\phi_2$ and speed $c_2$ for $t \rightarrow +\infty$}

This section shall be concerned with the limit behaviour of the solution $u(t,x)$ for $t \rightarrow + \infty$.

We will show the following theorem:

\begin{Theorem} \label{Thm_connection_two_traveling_fronts}
	Let $f$ satisfy \eqref{bounds_on_f}, \eqref{F_1}-\eqref{F_8} and let $u(t,x)$ be the unique entire solution of (\ref{DiffEqu}) such that $0 < u(t,x) <1$ and $\partial_t u(t,x) >0$ for all $(t,x) \in \R \times \bar{D}$ and
	\begin{align} \label{boundary_condition_for_negative_times}
	u(t,x)- \phi_1(x_1+c_1 t) \rightarrow 0 \text{ as } t \rightarrow -\infty \text{ uniformly in } x \in \bar{D},
	\end{align}
	then  
	\begin{align}
	u(t,x) - \phi_2(x_1 + c_2t + \beta) \rightarrow 0  \text{ as } t \rightarrow +\infty \text{ uniformly in } x \in \bar{D},
	\end{align}
	where $\beta \in \R$.
\end{Theorem}

As announced in the summary of the strategy given in the introduction, we start by showing the following lower bound.

\begin{Lemma}[Lower bound. The less involved direction]      \label{Lemma_subsolution}
	Under the assumptions of Theorem \ref{Thm_connection_two_traveling_fronts} there exists a time $T^->0$ and constants $\beta^-, C^-, \omega >0$ such that 
	\begin{align}
	\max \{	\phi_2(x_1+c_2t - \beta^-) -C^- e^{-\omega t},0 \} \leq u(t,x)
	\end{align}
	for $x \in D$ and $t \geq T^-$ .
\end{Lemma}

\begin{proof}[Proof of the Lemma]
	The proof follows the idea found in the proof of Theorem 7.1 in \cite{MatanoObst} or the respective proofs in \cite{MatanoStab} and \cite{FifeMcLeod}, i.e. constructing a suitable subsolution as follows.
	First of all, by assumption \eqref{boundary_condition_for_negative_times}, for every $\lambda \in (0,1)$ there is $t_\lambda \in \R$ such that 
	\begin{align*}
	|u(t_\lambda,x) - \phi_1(x_1+c_1 t_\lambda) | \leq \frac{\lambda}{2} \text{ for all } x \in D
	\end{align*}
	and $\zeta_+>0$ such that 
	\begin{align*}
	\phi_1(x_1+c_1 t_\lambda) \geq 1-\frac{\lambda}{2} \text{ and } \phi_2(x_1+c_2 t_\lambda) \geq 1-\frac{\lambda}{2} \text{ for all } x_1 \geq \zeta_+ .
	\end{align*}
		Employing $\partial_t u >0$ and $\phi_1', \phi_2' >0$ this implies that
	\begin{align}
	1 \geq u(t,x) \geq 1 - \lambda \text{ and } 1 \geq \phi_1(x_1+c_1 t), \phi_2(x_1+c_2 t ) \geq 1- \frac{\lambda}{2} 
	\end{align}
	  for all  $x_1 \geq \zeta_+$  and $t \geq t_\lambda$.
	From this we can conclude that
	\begin{align*}
	| \phi_1(x_1+c_1t) - u(t,x)| , | \phi_2(x_1+c_2t) -u(t,x)| \leq \lambda \text{ for } x_1 \geq \zeta_+  \text{ and } t \geq t_\lambda.
	\end{align*}
	There is also $\zeta_- <0$ such that
	\begin{align*}
	0 < \phi_2(x_1+c_2 t_\lambda) , \phi_1(x_1+c_1t_\lambda) \leq \frac{\lambda}{2} \text{ for } x_1 \leq \zeta_- .
	\end{align*}

	Summing this up, we know that outside a compact set $u(t,x)$ is at time $t_\lambda$ already close to $\phi_2(x_1+c_2t)$, i.e
	\begin{align*}
	|u(t_\lambda,x) -\phi_2(x_1+c_2 t_\lambda) | \leq \lambda \text{ for all } x \in \bar{D} \setminus ([\zeta_-, \zeta_+] \times \bar{\Omega}).
	\end{align*}
	
	Let us set furthermore
	\begin{align*}
	\omega := \min \left ( \frac{|f_1'(0)|}{4}, \frac{|f_1'(1)|}{4}, \frac{|f_2'(0)|}{4}, \frac{|f_2'(1)|}{4}, 1      \right ) 
	\end{align*}

	and choose $\rho >0$ such that
	\begin{align*}
	\text{ for } i \in \{ 1,2\}
	\begin{cases}
	|f_i'(s) -f_i'(0)| \leq \omega  &\text{ for all } s \in [0,\rho] \\
	|f_i'(s) -f_i'(1)| \leq \omega &\text{ for all } s \in [1-\rho,1]
	\end{cases} \quad .
	\end{align*}
	
	Let $A^->0$ be such that 
	\begin{align*}
	\begin{cases}
	\phi_2(z) \geq 1-\frac{\rho}{2} &\text{ for all } z \geq A^-, \\
	\phi_2(z) \leq \frac{\rho}{2} &\text{ for all } z \leq -A^- .
	\end{cases}
	\end{align*}
	
	Since $\phi_2'$ is positive and continuous on $\R$ we have
	\begin{align*}
	\delta_2^- := \min \limits_{z \in [-A^-, A^-]} \phi_2'(z) >0  . 
	\end{align*}
	
	Let us set
	
	\begin{align*}
	\mu_2^-:= \min \left \{  \frac{\rho}{2}, \frac{1}{2} \right \} .
	\end{align*}
	Choose now 
	\begin{align}
	\lambda = \mu_2^- .
	\end{align}
	Since $u$ is continuous and $0<u(t,x) <1$  for all $(t,x) \in \R \times \bar{D}$, it follows that
	\begin{align*}
	\min \limits_{x \in [\zeta_-, \zeta_+ ]\times \bar{\Omega}} u(t_\lambda,x) >0 
	\end{align*}
	and therefore there is $\beta^->0$
	\begin{align}
	\phi_2(x_1+c_2t_\lambda -\beta^-)\leq u(t_\lambda,x) \quad \text{ for all } x \in [\zeta_-,\zeta_+] \times \bar{\Omega}
	\end{align}
	because $\lim \limits_{z \rightarrow -\infty} \phi_2(z) =0$.
	Let now 
	\begin{align}
	\tilde{u}(t,x) := u(t-t_\lambda,x) \quad \text{ for } t \geq 0 , x \in D
	\end{align}
	and let us define the following auxiliary functions
	\begin{align}
	v^-(t) &:= \mu_2^- \exp(-\omega t) &&\text{ for } t \geq 0 \\
	V_2^-(t) &:= 4 \|f_2'\|_{L^\infty} \frac{1}{\delta_2^- \omega} \mu_2^- \exp(-\omega t) &&\text{ for } t \geq 0 .
	\end{align}
	
	For later reference let us point out that $v^-$ satisfies the following differential equation
	\begin{align} \label{Diffequ_of_v_minus}
	\dot{v}^-(t) = -\omega v^-(t) \text{ for all } t \geq 0 .
	\end{align}
	Now we have all the auxiliary functions in place to state our candidate for a subsolution
	\begin{align}
	u_2^-(t,x) := \max \{ \phi_2(\xi_2^-(t,x)) -v^-(t),0\} \text{ for } t \geq 0,
	\end{align}
	where $\xi_2^-$ are perturbed moving frame coordinates given by
	\begin{align}
	\xi_2^-(t,x) := x_1 + c_2(t+t_\lambda) - \beta^- + V_2^-(t) -V_2^-(0) \text{ for } t \geq 0, x \in \bar{D}.
	\end{align}
	
	First of all let us check that 
	\begin{align}
	u_2^-(0,x) \leq \tilde{u}(0,x) \text{ for all } x \in \bar{D}
	\end{align}
	by choice of $\beta^-$ and $\mu_2^-$.
	Either we are in the case $x \in [\zeta_-,\zeta_+] \times \bar{\Omega}$ and hence 
	\begin{align}
	u_2^-(0,x) \leq \phi_2(x_1+c_2t_\lambda-\beta^-) \leq \tilde{u}(0,x) = u(t_\lambda,x)
	\end{align}
	by choice of $\beta^-$ or
	$x \in \bar{D} \setminus [\zeta_-,\zeta_+] \times \bar{\Omega}$ and hence 
	\begin{align}
	u_2^-(0,x) &\leq \max \{\phi_2(x_1+c_2 t_\lambda) - \mu_2^- ,0 \} = \max \{\phi_2(x_1+c_2 t_\lambda) - \lambda ,0 \} \\
	 &\leq \tilde{u}(0,x) = u(t_\lambda,x) 
	\end{align}
	because $\phi_2' >0$ and by choice of $\zeta_-,\zeta_+$.
	Since $u_2^-$ was chosen independent of the lateral coordinates of our cylinder, $y \in \bar{\Omega}$, it meets the required Neumann boundary condition
	\begin{align}
	\frac{\partial}{\partial \nu} u_2^-(t,x) = 0 \text{ on } [0,\infty) \times \partial D. 
	\end{align}
	To be able to compare $u_2^-$ and $\tilde{u}$, it remains to show that $u_2^-$ is a subsolution, i.e.
	
	\begin{align}
	\sL u_2^- = \frac{\partial}{\partial t}u_2^- - \Delta u_2^- -f(x,u_2^-) \leq 0 \text{ for all } t \geq 0, x \in D .
	\end{align}
	It suffices to check this on the set where $\phi_2(\xi_2^-(t,x))-v^-(t) >0$ as $\max \{ v_1,v_2\} $ is a subsolution, whenever $v_1$, $v_2$ are subsolutions. Let us therefore restrict ourselves to that set in the following.
	We can estimate $\sL$ as 
	\begin{align}
	\sL u_2^- &= \dot{\xi}_2^-(t,x) \phi_2'(\xi_2^-) - \dot{v}^-(t) - \phi_2''(\xi_2^-) -f(x,\phi_2(\xi_2^-)-v^-(t)) \\
	&= (c_2 +\dot{V}^-_2(t)) \phi_2'(\xi_2^-) - \dot{v}^-(t) - \phi_2''(\xi_2^-) -f(x, \phi_2(\xi_2^-)-v^-(t)) \\
	&= \dot{V}_2^-(t) \phi_2'(\xi_2^-) - \dot{v}^-(t) +f_2(\phi_2(\xi_2^-)) -f(x,\phi_2(\xi_2^-)-v^-(t)) \\
	&\leq \dot{V}_2^-(t) \phi_2'(\xi_2^-) - \dot{v}^-(t) +f_2(\phi_2(\xi_2^-)) -f_2(\phi_2(\xi_2^-)-v^-(t)) .
	\end{align}
	
	
	We will now distinguish between the cases $\xi_2^- < A^-, \xi_2 \in [-A^-,A^-]$ and $\xi_2^- > A^-$. 
	If $\xi_2^- < -A^-$ we have $\phi_2(\xi_2^-) -v^-(t) \leq \phi_2(\xi_2^-) \leq \rho $. Recalling that $\phi_2' >0, \dot{V}_2^-<0$, \eqref{Diffequ_of_v_minus} and involving the mean value theorem, we find
	\begin{align}
	\sL u_2^- &\leq (f_2'(0)+\omega) v^-(t)-\dot{v}^-(t) \\
	&=(f_2'(0)+\omega) v^-(t) + \omega v^-(t) \\
	&\leq 0
	\end{align}
	by the choice of $\omega$.
	If $\xi_2^-> A^-$ we have $\phi_2(\xi_2^-) >\phi_2(\xi_2^-) -v^-(t) \geq 1 -\frac{\rho}{2} -\mu_2^- \geq 1-\rho$ by the choice of $\mu_2^-$. Using again $\phi_2' >0, \dot{V}_2^- <0$, \eqref{Diffequ_of_v_minus} and the mean value theorem it holds
	\begin{align}
	\sL u_2^- \leq (f_2'(1)+\omega) v^-(t)- \dot{v}^-(t) =(f_2'(1)+2 \omega ) v^-(t) \leq 0 
	\end{align}
	again by the choice of $\omega$.
	In the last case, if $\xi_2^- \in [-A^-,A^-]$, by definition of $\delta_2^-$ and $\dot{V}_2^-<0$ it holds
	\begin{align}
	\sL u_2^- &\leq \| f_2' \|_{L^\infty} v^-(t)-\dot{v}^-(t)+ \delta_2^- \dot{V}_2^-(t) \\
	&= ( \|f_2' \|_{L^\infty}+\omega) v^-(t)  + \delta_2^- \dot{V}_2^-(t)  \\
	&\leq (2 \|f_2'\|_{L^\infty}-4 \|f_2'\|_{L^\infty}) v^-(t) \leq 0
	\end{align}
	
	Since we have shown that $u_2^-$ is a subsolution, it follows from the comparison principle that
	\begin{align}
	u_2^- (t,x) \leq \tilde{u}(t,x) \text{ for all } t \geq 0 \text{ and } x \in \bar{D}.
	\end{align}

	Replacing $\beta-$ with $\beta^- + V^-(0)$ finishes the proof.
\end{proof}

In exactly the same way can we estimate the solution from above against a traveling wave with profile $\phi_1$ and speed $c_1$ as stated in

\begin{Lemma}  \label{Lemma_supersolution_phi_1}
	Under the assumptions of Theorem \ref{Thm_connection_two_traveling_fronts} there exists a time $T^+>T^-$ and constants $\beta_1^+, C_1^+, \omega >0$ such that 
	\begin{align}
	\min \{	\phi_1(x_1+c_1t + \beta_1^+) +C_1^+ e^{-\omega t} ,1 \} \geq u(t,x)
	\end{align}
	for $x \in D$ and $t \geq T^+$, where $\omega$ and $T^-$ are as in Lemma \ref{Lemma_subsolution}.
\end{Lemma}

\begin{proof}
	The proof follows exactly the classical strategy of the proof of Lemma \ref{Lemma_subsolution}. One only needs to exploit that $f(\cdot,u) \leq f_1(u)$ for all $u \in [0,1]$.
\end{proof}

%

With these Lemmata we are in the position to state the more involved direction, i.e. the estimation against a traveling wave with profile $\phi_2$ and speed $c_2$ from above.

\begin{Lemma} \label{Lemma_supersolution_2}
	Under the assumptions of Theorem \ref{Thm_connection_two_traveling_fronts} there exists a time $T>\max\{T^-,T^+\}$ and constants $\beta_2^+, C_2^+, \eta >0$ such that 
	\begin{align}
\min \set {	\phi_2(x_1+c_2t + \beta_2^+) +C_2^+ e^{-\eta t}, 1 } \geq u(t,x)
	\end{align}
	for $x \in D$ and $t \geq T$, where  $T^-$ and $T^+$ are as in Lemmata \ref{Lemma_subsolution} and \ref{Lemma_supersolution_phi_1}.
\end{Lemma}

\begin{proof}
	Let us start the proof with the statement of known facts on one-dimensional traveling waves (see e.g. \cite{FifeMcLeod}) and the statement of some auxiliary constants and functions.
	Linearising \eqref{ODE_travelling_wave} for $\phi_2$ around the stationary point $\phi_2 = 1$ and using that $f_2'(1)<0$ tells us that there is $C_{\phi_2
	}>0$ and $\lambda >0$ such that
	\begin{align} \label{phi_2_asymptotics}
	|1-\phi_2(z)| \leq  C_{\phi_2}e^{-\lambda z} \text{ for all } z \in \R.
	\end{align}
	We assume that $0<\omega < \lambda c_2$. Otherwise slightly decrease $\omega>0$.

	From \eqref{F_1}, \eqref{F_2}, it immediately follows that there is $C_f >0$ such that
	\begin{align}
	|f_1(u)- f_2(u)| \leq C_f |1-u| \text{ for all } u \in [0,1].
	\end{align}
	Let $\omega$, $\rho$ and $A^-$ be as in the proof of Lemma \ref{Lemma_subsolution}.  Let $0<\gamma < \frac{\rho}{4}$ be arbitrary and let $T > \max\{T^-, T^+\}$ be such that
	\begin{align}
	\max\{C^-,C^+\}  ~e^{-\omega T} \leq \frac{\gamma}{2} ~,~  A^- -c_2T \leq -x_0 ~\text{  and } \frac{C_f C_{\phi_2}}{\lambda c_2} e^{\lambda (x_0 -c_2 T)} \leq \frac{\rho}{4}.
	\end{align}
	The second assumption means that we wait until only the tail of the wave will lie in the portion of $D$ where $x_1 \geq -x_0$.
	Let furthermore $\tilde{\zeta}_-<0$ be such that
	\begin{align}
	\phi_1(x_1+c_1T + \beta_1^+) \leq \frac{\gamma}{2} \text{ and } \phi_2(x_1+c_2T) \leq \gamma \text{ for } x_1 \leq \tilde{\zeta}_- 
	\end{align} 
	and $\tilde{\zeta}_+>0$ such that
	\begin{align}
	\phi_2(x_1+c_2 T- \beta^-) \geq 1-\gamma  \text{ and hence also } \phi_2(x_1+c_2T ) \geq 1-\gamma \text{ for } x_1 \geq \tilde{\zeta}_+ .
	\end{align}
	
	Since $\lim \limits_{z \rightarrow + \infty} \phi_2(z) = 1$ and $\max \limits_{x \in [\tilde{\zeta}_-, \tilde{\zeta}_+]\times \bar{\Omega}} u(T,x)<1$, there is $\beta >0$ such that
	\begin{align}
	\phi_2(x_1+c_2T + \beta) \geq u(T,x) \text{ in } [\tilde{\zeta}_-,\tilde{\zeta}_+] \times \bar{\Omega}.
	\end{align}
	Finally let us define
	\begin{align}
	\delta_2^+:=	\min \limits_{z \in [-A^-,A^-]} \phi_2'(z) >0 .
	\end{align}
	Let us now define the following auxiliary functions:
	\begin{align}
	v_2^+(t) &:= \left ( \gamma - \frac{C_f C_{\phi_2}\exp(\lambda (x_0-c_2T))}{\omega-\lambda c_2} \right ) e^{-\omega (t-T)} \\
	&\quad +\frac{C_f C_{\phi_2}\exp(\lambda (x_0-c_2T))}{\omega-\lambda c_2} e^{-\lambda c_2 (t-T)} 
	\end{align}
	for all $ t \geq T$ and
	\begin{align}
	V_2^+(t) &:=  \frac{\|f_2'\|_{L^\infty}+\omega}{\delta_2^+} \int \limits_T^t v(\tau) \dx{\tau} .
	\end{align}
	For later reference, let us note that
	\begin{align} \label{properties_of_v2+}
	\begin{split}
	0 &\leq v_2^+(t) \leq \frac{\rho}{2} \text{ for all } t \geq T ~,~ v_2^+(T)= \gamma  \text{ and } \\
	\dot{v}_2^+(t) &= -\omega v_2^+(t) + C_f C_{\phi_2}\exp(\lambda (x_0-c_2T)) e^{-\lambda c_2 (t-T)}  \text{ for all }  t \geq T.
	\end{split}
	\end{align}
	(That $v_2^+(T)= \gamma$ is obvious. That $0 \leq v_2^+(t)$ for $t\geq T$ can be seen by directly calculating its zeros and employing $\omega < \lambda c_2$ and with this property, the differential equation solved by $v_2^+$ and the choice of $T$ it is easy to see that $v_2^+ \leq \frac{\rho}{2}$.)
	
	Now we have everything in place to state our candidate for the supersolution, namely
	\begin{align}
	u^+_2(t,x) := \min \{\phi_2(\xi_2^+) +v_2^+(t),1 \},
	\end{align}
	where $\xi_2^+(t,x) := x_1+c_2t +\beta +V_2^+(t)$.
	By construction $u_2^+$ does satisfy
	\begin{align}
	\frac{\partial u_2^+}{\partial \nu} = 0 \text{ on } \partial D \times [0,\infty) .
	\end{align}
	For the initial time $T$ we have that either $x \in D \setminus \bra { [\tilde{\zeta}_-,\tilde{\zeta}_+] \times \bar{\Omega} }$ and then 
	\begin{align}
	u(T,x) \leq \min \{\phi_2(x_1+c_2T)+ \gamma ,1\}= \min \{ \phi_2(x_1+c_2T) +v_2^+(T),1\} \leq u_2^+(T,x)
	\end{align}
	or we are in the case that $x \in [\tilde{\zeta}_-,\tilde{\zeta}_+] \times \bar{\Omega}$, then by choice of $\beta>0$ it holds that
	\begin{align}
	u(T,x) \leq \phi_2(x_1+c_2T+\beta) = \phi_2(\xi(T,x)) \leq u_2^+(T,x) .
	\end{align}
	
	It remains to show that $u_2^+$ is indeed a supersolution, for the operator $\sL$, i.e that
	\begin{align}
	\sL u_2^+ = \frac{\partial}{\partial t} u_2^+ - \Delta u_2^+ -f(x,u_2^+) \geq 0 \text{ in } D \text{ for all } t \geq T.
	\end{align}
	It is sufficient to check this in the set $\{u_2^+ <1\}$ ( since the minimum of supersolutions is again a supersolution) therefore we will restrict ourselves to this set in the following.
	We can estimate as follows
	\begin{align}
	\sL u_2^+ &= \dot{\xi}_2^+ \phi_2'(\xi_2^+) +\dot{v}_2^+ -\phi_2''(\xi_2^+) -f(x,\phi_2(\xi_2^+)+v_2^+ (t)) \\
	&= \dot{V}_2^+(t) \phi_2'(\xi_2^+)  + \dot{v}_2^+(t) +f_2(\phi_2(\xi_2^+)) -f(x, \phi_2(\xi_2^+)+v_2^+(t)) \\
	&\geq \begin{cases}
	\dot{V}_2^+(t) \phi_2'(\xi_2^+)  + \dot{v}_2^+(t) +f_2(\phi_2(\xi_2^+)) -f_2(\phi_2(\xi_2^+)+v_2^+(t)) &, x_1 \leq -x_0 \\
	\dot{V}_2^+(t) \phi_2'(\xi_2^+)  + \dot{v}_2^+(t) +f_2(\phi_2(\xi_2^+)) -f_1(\phi_2(\xi_2^+)+v_2^+(t)) &, x_1 > -x_0 
	\end{cases} \\
	&\geq \begin{cases}
	\dot{V}_2^+(t) \phi_2'(\xi_2^+)  + \dot{v}_2^+(t) +f_2(\phi_2(\xi_2^+)) -f_2(\phi_2(\xi_2^+)+v_2^+(t)) \quad , x_1 \leq -x_0 \\
	\dot{V}_2^+(t) \phi_2'(\xi_2^+)  + \dot{v}_2^+(t) +f_1(\phi_2(\xi_2^+)) -f_1(\phi_2(\xi_2^+)+v_2^+(t)) -C_f |1-\phi_2(\xi_2^+)| 
	\\ \quad , x_1 > -x_0 
	\end{cases}
	\end{align}
	Let us distinguish, as in the proof of Lemma \ref{Lemma_subsolution} between the cases $\xi_2^+ > A^-$,  $\xi_2^+ \in [-A^-,A^-]$ and $\xi_2^+ < -A^-$.
	
	In the case $\xi_2^+ >A^-$ we have to treat the case $x_1 \geq -x_0$ as well as the case $x_1 \leq -x_0$.
	If $x_1 \geq -x_0$ we can use \eqref{phi_2_asymptotics} and can then estimate in this case 
	\begin{align}
	\sL u_2^+ \geq \dot{V}_2^+(t) \phi_2'(\xi_2^+) + \dot{v}_2^+(t) -f_1'(\sigma) v_2^+(t) -C_f C_{\phi_2} \exp(\lambda (x_0-c_2T)) e^{-\lambda c_2 (t-T)} \\
	\geq -(f_1'(\sigma)+\omega) v_2^+(t) \geq 0
	,
	\end{align}
	where $\sigma \in (1-\frac{\rho}{2},1)$ comes from the mean value theorem and we have used the definition of $A^-$ and $\omega$, \eqref{properties_of_v2+}, \eqref{F_3} and that $\phi_2', \dot{V}_2^+ \geq 0$.
	
	If $\xi_2^+ > A^-$ and $x_1 \leq -x_0$ then we can estimate
	\begin{align}
	\sL u_2^+ &\geq \dot{V}_2^+(t) \phi_2'(\xi_2^+) + \dot{v}_2^+ -f_2'(\tilde{\sigma}) v^+_2(t) \\
	&= \dot{V}_2^+(t) \phi_2(\xi_2^+) -\omega v_2^+(t) + C_f C_{\phi_2} \exp(\lambda(x_0-c_2T)) e^{-\lambda c_2(t-T)} \\
	&\quad - (f_2'(\tilde{\sigma})+\omega) v_2^+(t) \\
	&\geq - (f_2'(\tilde{\sigma})+\omega) v_2^+(t)  \geq 0,
	\end{align}
	where $\tilde{\sigma} \in (1-\frac{\rho}{2},1)$ comes again from the mean value theorem and the rest follows as before.
	
	In the case $\xi_2^+ \in [-A^-,A^-]$ we are by choice of $T$ always in that portion of $D$ where $x_1 \leq -x_0$ and hence we can estimate
	\begin{align}
	\sL u_2^+ &\geq \dot{V}_2^+(t) \phi'_2(\xi_2^+)+\dot{v}_2^+(t) +f_2(\phi_2(\xi_2^+)) -f_2(\phi_2(\xi_2^+)+v_2^+(t))   \\
	&\geq \dot{V}_2^+(t) \phi_2'(\xi_2^+) + \dot{v}_2^+(t) - \|f_2'\|_{L^\infty} v_2^+(t) \\
	&\geq \dot{V}_2^+(t) \delta_2^+ + \dot{v}_2^+(t) - \|f_2'\|_{L^\infty} v_2^+(t) \\
	&= \dot{V}_2^+(t) \delta_2^+ - \omega v_2^+(t)  + C_f C_{\phi_2} \exp(\lambda(x_0-c_2T)) e^{-\lambda c_2(t-T)} - \|f_2'\|_{L^\infty} v_2^+(t) \\
	&\geq \dot{V}_2^+(t) \delta_2^+ -(\omega +\| f_2' \|_{L^\infty}) v_2^+(t) \\
	&= 0 ,
	\end{align}
	by definition of $V_2^+$ and \eqref{properties_of_v2+}.
	It remains to look into the last case $\xi_2^+ <-A^-$. In this regime again by choice of $T$ we are in the portion of $D$ where $x_1 \leq -x_0$ and hence we can estimate
	\begin{align}
	\sL u_2^+ &\geq \dot{V}_2^+(t) \phi_2(\xi_2^+) +\dot{v}_2^+(t) +f_2(\phi_2(\xi_2^+)) - f_2(\phi_2(\xi_2^+)+v_2^+(t)) \\
	&= \dot{V}_2^+(t) \phi_2'(\xi_2^+) - \omega v_2^+(t) + C_f C_{\phi_2} \exp(\lambda(x_0-c_2T)) e^{-\lambda c_2(t-T)} -f_2'(\bar{\sigma}) v_2^+(t) \\
	&\geq -(f_2'(\bar{\sigma}) +\omega )v_2^+(t) \geq 0 ,
	\end{align}
	where $\bar{\sigma}\in (0, \rho)$ comes from the mean value theorem and we have used the definition of $A^-$ and $\omega$, \eqref{properties_of_v2+}, \eqref{F_3} and that $\phi_2', \dot{V}_2^+ \geq 0$.
	Thereby we have shown that $u_2^+$ has the claimed property of being a supersolution.
	
	This concludes the proof of the Lemma.
\end{proof}

\subsection{approach to a translation of $\phi_2(x_1-c_2 t)$}

The strategy we are going to follow in this section is the classical one, as it is found e.g. in \cite{FifeMcLeod}.

We will need the following auxiliary Lemmata.

\begin{Lemma}[stability] \label{lemma_stability}
	Let $u$ be a solution of \eqref{DiffEqu} that at a time $t_0 >T$ is already close to a traveling wave $\phi_2(x_1+c_2t+\beta)$ for some $\beta \in \R$ i.e. 
	\begin{align}
	|u(t_0,x) -\phi_2(x_1+c_2t_0 + \beta) | \leq \epsilon \text{ for all } x \in D
	\end{align}
	where $0 < \epsilon < \frac{\rho}{4}$ then, it holds for all $t \geq t_0$ and $x \in D$ that
	\begin{align}
	|u(t,x)-\phi_2(x_1+c_2t+\beta)| \leq \delta(\epsilon,t_0) ,
	\end{align}
	where $\delta(\epsilon, t_0) \searrow 0$ as $\epsilon \searrow 0$ and $t_0 \nearrow +\infty$. $T$ and $\rho$ are as in the proof of Lemma \ref{Lemma_supersolution_2}.
	
\end{Lemma}

\begin{proof}
	
	Note that unlike in the stability result in \cite{FifeMcLeod}, it is not sufficient for the solution of \eqref{DiffEqu} to once be close to a traveling wave in order to stay close indefinitely. The reason is that the tail of the wave will always lie in a region where $f(x,\cdot) = f_1(\cdot)$ and will therefore introduce a disturbance that enters in the form of a possible shift. But since this possible shift is integrable, we can make sure that we do not get driven too far from $\phi_2(x_1+c_2t + \beta) $ if we start late enough and thereby do not accumulate too much of the disturbance. 
	
	Let us now turn to the formalities of the proof. It consists of revisiting the Lemmata \ref{Lemma_subsolution} and \ref{Lemma_supersolution_2}. If $t_0$ is large enough such that only the tail of $\phi_2(x_1+c_2+\beta)$ lies right of $x_1=-x_0$, i.e. $T<t_0$, we know that
	\begin{align}
	u(t,x) \leq \phi_2(x_1+c_2 t + \beta + V_2^+(t)) + v_2^+(t)
	\end{align}	
	where 
	\begin{align}
	v_2^+(t) = (\epsilon+C(t_0)) e^{-\omega(t-t_0)} - C(t_0) e^{-\lambda c_2(t-t_0)} \leq \epsilon +2 C(t_0)
	\end{align}
	and $C(t_0) \searrow 0$ as $t_0 \nearrow + \infty$.
	\begin{align}
	V_2^+(t) = C \int \limits_{t_0}^t v_2^+(\tau) \dx{\tau} \leq C \left (  \frac{\epsilon}{\omega} +C(t_0) \left ( \frac{1}{\omega} + \frac{1}{\lambda c_2}   \right )   \right) = C (\epsilon+ C(t_0))
	\end{align}
	
	Therefore we know that 
	\begin{align}
	u(t,x) - \phi_2(x_1+c_2+\beta) &\leq \phi_2(x_1+c_2 t +\beta +V_2^+(t)) - \phi_2(x_1+c_2 t+\beta) +v_2^+(t)  \\
	&\leq \| \phi_2'   \|_{L^\infty} V_2^+(t) + v_2^+(t) = C(\epsilon + C(t_0))
	\end{align}
	
	From Lemma \ref{Lemma_subsolution} we know that 
	\begin{align}
	u(t,x) &\geq \phi_2(x_1+c_2 t +\beta +V_2^-(t)-V_2^-(t_0)) -v_2^-(t) \\
	&\geq \phi_2(x_1+c_2 t + \beta -V_2^-(t_0)) -v_2^-(t) , 
	\end{align}
	where $0 \leq V_2^-(t_0) \leq C \epsilon$ and $0 \leq v_2^-(t) \leq C \epsilon$ and hence
	\begin{align}
	u(t,x) - \phi_2(x_1+c_2 t + \beta) &\geq \phi_2(x_1+c_2 t + \beta - V_2 ^-(t_0) ) - \phi_2(x_1+c_2 t + \beta) -v_2^-(t) \\
	&\geq -\| \phi_2' \|_{L^\infty} V_2^-(t_0) -v_2^-(t) \geq -C \epsilon .
	\end{align} 
	Summing it all up, we have that
	\begin{align}
	| u(t,x) - \phi_2(x_1+c_2 t + \beta)  | \leq C ( \epsilon + C(t_0)) .
	\end{align}
	Since $C(t_0) \searrow 0$ for $t_0 \nearrow + \infty$ this was to be proven.

\end{proof}

From here on it will be more convenient to work in moving frame coordinates $(z,y)$ where $z = x_1 +c_2t$. In the new coordinates $u$ solves
\begin{align}
\partial_t u + c_2 \partial_z u - \Delta_{z,y} u = f((z-c_2t,y),u)
\end{align}

\begin{Lemma} \label{apriori_bounds_u}
	There is $\sigma >0$ with $\sigma > \frac{|c_2|}{2}$ and $C>0$ such that
	\begin{align}
	|1-u|, |\nabla_{z,y} u|, |D^2_{z,y} u|, |\partial_t u | < C (e^{(\frac{1}{2}c_2-\sigma)z} + e^{-\eta t}) &\quad ,z>0 \\
	|u|, |\nabla_{z,y} u|, |D^2_{z,y} u|, |\partial_t u | < C (e^{(\frac{1}{2}c_2+\sigma)z} + e^{-\eta t}) &\quad ,z<0 
	\end{align}
	(where we always have omitted the arguments $(t,z,y)$.)
\end{Lemma}

\begin{proof}	
	We are following the proof of Lemma 4.3 in \cite{FifeMcLeod}.
	It is well known (/can be seen by linearizations around $1$ and $0$) that the wave-front $\phi_2$ approaches $1$ and $0$ exponentially.
	E.g. the linearisation around  $\phi_2 =1$ shows that $\phi_2(z) \rightarrow 1$ for $z \rightarrow +\infty$ with approximately the rate
	\begin{align}
	\exp \left ( \frac{1}{2} \left ( c_2 - \sqrt{c_2^2-4 f_2'(1)} \right )z \right)
	\end{align}
	For $ z \rightarrow - \infty$ one gets a similar result.
	Together with Lemmata \ref{Lemma_subsolution} and \ref{Lemma_supersolution_2} we find:
	\begin{align} \label{bound_zeroth_order}
	\begin{split}
	|u(t,z,y) | &\leq \phi_2(z+\beta_2^+) +C_2^+ e^{-\eta t}  \\
	&\leq C \left ( \exp \left ( \left ( \frac{1}{2} c_2 + \sigma \right ) z \right ) + e^{-\eta t} \right ) \text{ for } z <0 \text{ and } \\
	|1- u(t,z,y) | &\leq 1 - \left (   \phi_2  ( z-\beta^- ) - C^- e^{-\omega t}         \right ) \\
	&\leq C \left ( \exp \left ( \left ( \frac{1}{2} c_2 - \sigma \right ) z \right ) + e^{-\omega t} \right ) \text{ for } z >0
	\end{split}
	\end{align}
	
	Since $f$ is Lipschitz, there is $L>0$ such that 
	\begin{align}
	|f(x,u)| \leq L |u| \text{ and } |f(x,u) | \leq L |1-u| \text{ for } u \in [0,1] \text{ and all } x \in D.
	\end{align}
	This together with \eqref{bound_zeroth_order} implies
	\begin{align}
	|f((z,y),u(t,z,y))| \leq C \left (   \exp \left (    \frac{1}{2} c_2 z - \sigma |z|                                       \right )      +e^{-\eta t}                                  \right ) .
	\end{align}
	For the higher order estimates we employ Schauder Theory (e.g \cite{friedman2008partial} Thm 5 Chap 3 and Thm 4 in Chap 7 for the a-priori bound on the Hölder-norm of $f(u)$).  Hence it does also hold:
	\begin{align}
	|\nabla_{z,y} u | , |D^2_{z,y} u | , |\partial_t u| \leq C \left (   \exp \left (    \frac{1}{2} c_2 z - \sigma |z|                                       \right )      +e^{-\eta t}                                  \right ) .
	\end{align}

\end{proof}

Now we have everything in place to proof Theorem \ref{Thm_connection_two_traveling_fronts}.

\begin{proof}[Proof of Theorem \ref{Thm_connection_two_traveling_fronts}]
	
	For the identification of the limit equation in the moving frame we will use a multidimensional analogon of a Lyapunov function argument given in \cite{FifeMcLeod}. Lyapunov functions are a well known and very helpful tool for investigating the long-term behaviour of parabolic partial differential equations (see e.g. \cite{Eberle2017142}).
	
	For the sake of completeness we will repeat the arguments given in \cite{FifeMcLeod} and add the slight modifications we had to make.
	Let us define the Lyapunov function as 
	\begin{align}
	\cL[u](t) := \int \limits_D e^{-c_2 z} \left (  \frac{1}{2} | \nabla_{z,y} u |^2 -F(u) +H(z) F(1)     \right ) \dx{z} \dx{y} ,
	\end{align}
	where $F(s) := \int_0^s f_2(\sigma) \dx{\sigma}$ and $H$ is the heaviside-function.
	

	To ensure integrability in the definition of $\cL$ we cut $u$ off as follows
	\begin{align}
	w(t,z,y) &= u(t,z,y) &&\text{ for } |z| \leq m t, \\
	w(t,z,y) &= 0 &&\text{ for } z \leq -m t -1, \\
	w(t,z,y) &=1 &&\text{ for } z \geq m t +1,
	\end{align}
	for some $m >0$ to be specified later. And we assume $w$ to be smoothed out in a manner such that Lemma \ref{apriori_bounds_u} still holds for $w$.

	Employing Lemma \ref{apriori_bounds_u} we find that
	\begin{align}
	|\cL  [w]| \leq C \int \limits_\Omega \int \limits_{-m t -1}^{m t +1} e^{-c_2z} \left ( e^{c_2 z -2 \sigma |z| } + e^{-2\eta t} \right )  \dx{z} \dx{y}
	\end{align}
	which is uniformly bounded for all $t>0$ if  $m>0$ is chosen such that $c_2m -2\eta <0$. Let us choose $m$ such that $m < \frac{1}{2} \min \left \{ \frac{2 \eta}{c_2},c_2   \right \}$.
	
	Using integration by parts it follows
	\begin{align}
	\dot{\cL}[w](t) = - \int \limits_D e^{-c_2z} \left (  -c_2 \partial_z w + \Delta_{z,y} w +f_2(w)       \right ) \partial_t w  ~  \dx{z} \dx{y} .
	\end{align}
	
	Unfortunately, $w$ does not solve $\partial_t w = -c_2 \partial_z w + \Delta_{z,y} w + f_2(w)$ and we do not get a sign for $\dot{\cL}$. This is why we try to control the error against
	\begin{align}
	Q[w] = \int \limits_D e^{-c_2 z}  \left (  \Delta_ {z,y} w -c_2 \partial_z w +f_2(w)  \right )^2  \dx{z} \dx{y},
	\end{align}
	i.e.
	\begin{align} \label{error_derivative_Lyapunov_function}
	&\dot{\cL}[w](t) + Q[w](t) = \\ &- \int \limits_D e^{-c_2 z} \left (  - c_2 \partial_z w + \Delta_{z,y} w +f_2(w)   \right ) \left (  \partial_t w - \Delta_{z,y} w + c_2 \partial_z w -f_2(w)   \right )  \dx{z} \dx{y}
	\end{align}
	
	Note that for $|z| \leq mt$ $w$ solves
	\begin{align*}
	\partial_t w - \Delta_{z,y} w +c_2 \partial_z w -f_2(w) = f(z-c_2t,w) -f_2(w)
	\end{align*}
	and that $ f((z-c_2t,y),w) = f_2(w)$  if $ t\geq \frac{1+x_0}{c_2-m}$ in $\set {|z| \leq mt}$ .

	For $t \geq \frac{1+x_0}{c_2-m}$ the last factor in the integral in \eqref{error_derivative_Lyapunov_function} vanishes in $\set {|z| \leq mt}$ and for $  \set {|z| \in (mt,mt+1] }$ we can use the growth estimates from Lemma \ref{apriori_bounds_u}. With all that we can conclude that
	\begin{align}
	\lim \limits_{t \rightarrow \infty} | \dot{\cL}[w](t) +Q[w](t)| =0 .
	\end{align}
	Since $Q[w] \geq 0$ this implies that
	\begin{align}
	\limsup \limits_{t \rightarrow \infty} \dot{\cL}[w](t) \leq 0 .
	\end{align}
	Hence there must be a subsequence $(t_n)_{n \in \N}$, $t_n \rightarrow \infty$ for $n \rightarrow \infty$ such that
	\begin{align}
	\lim \limits_{n \rightarrow \infty} \dot{\cL}[w](t_n) =0
	\end{align}
	because otherwise $\cL[w]$ could not be uniformly bounded in $t$.
	Therefore it must hold along that subsequence
	\begin{align} \label{limit_of_Q_along_sequence}
	\lim \limits_{n \rightarrow \infty} Q[w](t_n) =0 .
	\end{align}
	
	By Lemma \ref{apriori_bounds_u} and an Arzela-Ascoli argument for a further subsequence (again denoted by $(t_n)_{n \in \N}$) there is a function $u_\infty$ such that:
	\begin{align}
	u(\cdot, t_n) \rightarrow u_\infty(\cdot) \quad \text{ for } n \rightarrow \infty \text{ in } C^2(D), \\
	w(\cdot, t_n) \rightarrow u_\infty(\cdot) \quad \text{ for } n \rightarrow \infty \text{ in } C^2(D) . 
	\end{align}
	Therefore  since $Q\geq 0$ and \eqref{limit_of_Q_along_sequence} for any finite interval $I \subset \R$:
	\begin{align}
	0 &\leftarrow \left (  \int \limits_{I \times \Omega}  e^{-c_2 z} \left ( \Delta_{z,y}w -c_2 \partial_z w +f_2(w) \right )^2  \dx{z} \dx{y}    \right )(t_n) \\
	&\rightarrow \int \limits_{I \times \Omega} e^{-c_2 z} \left (  \Delta_{z,y} u_\infty -c_2 \partial_z u_\infty +f_2(u_\infty)  \right )^2  \dx{z} \dx{y}
	\end{align}
	as $n \rightarrow \infty$. Hence $u_\infty$ solves
	\begin{align}
	\Delta_{z,y} u_\infty -c_2 \partial_z u_\infty + f_2(u_\infty) =0 \quad \text{ a.e. in } D \text{ and } \\
	\lim \limits_{z \rightarrow -\infty} u_\infty(z,y) = 0 , \lim  \limits_{z \rightarrow \infty} u_\infty(z,y) =1
	\end{align}
	
	By uniqueness of traveling fronts up to translation in $z$ (see e.g. \cite{BerestyckiNirenberg} Thm 7.1), there is $\beta \in \R$ such that
	\begin{align}
	u_\infty(z,y) = \phi_2(z+\beta) .
	\end{align}
	
	Now the stability result \ref{lemma_stability} implies that 
	\begin{align}
	u(t,z,y) \rightarrow \phi_2(z+ \beta) \text{ uniformly in } D \text{ as } t \rightarrow +\infty ,
	\end{align}
	not only for the special subsequence $(t_n)_{n \in \N}$.
	This was to be proven.
	
\end{proof}

\appendix

\section{Proof of Theorem \ref{existence_uniqueness}}

\begin{proof}[Proof of Theorem \ref{existence_uniqueness}]
	We follow the proof of existence and uniqueness given in \cite{MatanoObst}. See Theorem 2.1 therein. For the sake of completeness, we repeat it here with slight adaptations to our setting.
	
	We define the candidates for the super- and subsolution before the wave front encounters the transition zone at $x_1 \leq 0$ as
	
	\begin{align}
	w^+(t,x_1) = 
	\begin{cases}
	\min \{\phi_1(x_1+c_1 t + \xi(t) ) + \phi_1(-x_1 +c_1 t + \xi(t)),1\} &, x_1 \geq 0 \\
	\min \{  2 \phi_1(c_1+\xi(t)),1      \}  &, x_1 <0
	\end{cases}
	\end{align}
	
	and our candidate for the subsolution shall be given as
	\begin{align}
	w^-(t,x_1) = 
	\begin{cases}
	\max \{\phi_1(x_1+c_1 t - \xi(t) ) - \phi_1(-x_1 +c_1 t - \xi(t)),0\} &, x_1 \geq 0 \\
	0  &, x_1 <0
	\end{cases}
	\end{align}
	Here $\phi_i$ are solutions of \eqref{ODE_travelling_wave} normalized such that $\phi_i(0) = \theta_i$.
	
	In this definition $\xi(t)$ is the solution of the ordinary differential equation
	\begin{align} \label{Appendix:equ:xi}
	\dot{\xi}(t) = M e^{\lambda (c_1+ \xi)} , t <-T , \quad \xi(-\infty)=0,
	\end{align}
	where $M$ and $T$ will be chosen later,  $\lambda$ is the positive root of
	\begin{align} \label{Appendix:equ:lambda}
	\lambda^2-c_1 \lambda+f_1'(0) =0 \text{ i.e. } \lambda = \left (  c_1+\sqrt{ c_1^2-4 f_1^2(0)  }   \right)
	\end{align}
	and
	\begin{align}
	\xi(t) = \frac{1}{\lambda} \log \left (   \frac{1}{1-c_1^{-1}M e^{\lambda c_1 t}}     \right) .
	\end{align}
	In order for this to be defined, we will need that 
	\begin{align}
	1-c_1^{-1}M e^{\lambda c_1 t} >0.
	\end{align}
	We also want that
	\begin{align}
	c_1 t + \xi(t) \leq 0 \text{ for  } -\infty <t \leq T
	\end{align}
	and therefore we set
	\begin{align}
	T:= \frac{1}{\lambda c_1} \log \left ( \frac{c_1}{c_1+M}   \right) 
	\end{align}
	
	We will now show for $M >0$ sufficiently large and a $T_1 \in (-\infty,T]$ that $w^+$ will be a super- and $w^-$ a subsolution of \eqref{DiffEqu} for $-\infty <t \leq T_1$.
	
	It is classical (see e.g. \cite{FifeMcLeod}) that there are positive constants $\alpha_0, \alpha_1, \beta_0, \beta_1$ such that
	\begin{align}
	\alpha_0 e^{\lambda z} &\leq \phi_1(z) \leq \beta_0 e^{\lambda z}  &, z \leq 0,  \label{Appendix:bound:phi1:zneg}\\
	\alpha_1 e^{-\mu z} &\leq 1- \phi_1(z) \leq \beta_1 e^{-\mu z} &, z >0 ,  \label{Appendix:bound:phi1:zpos}
	\end{align}
	where $\lambda$ is as in \eqref{Appendix:equ:lambda} and $\mu$ is given by
	\begin{align}
	\mu = \frac{1}{2} \left ( -c_1 + \sqrt{c_1^2-4 f_1'(1)}     \right) .
	\end{align}
	For the derivative $\phi_1'$ we have the same exponential behaviour
	\begin{align}
	\gamma_0 e^{\lambda z} &\leq \phi_1'(z) \leq \delta_0 e^{\lambda z} &, z \leq 0  , \label{Appendix:bound:phi1prime:zneg} \\
	\gamma_1 e^{-\mu z} &\leq \phi_1'(z) \leq \delta_1 e^{-\mu z} &, z >0 .   \label{Appendix:bound:phi1prime:zpos}
	\end{align}
	Furthermore, since $f_1$ was assumed to be $C^{1,1}$ we  have $L >0 $ such that
	\begin{align} \label{Appendix:C11bound}
	|f_1(u+v)-f_1(u)-f_1(v)| \leq L uv    \quad \text{ for } 0 \leq u,v \leq 1.
	\end{align}

	First of all, since $w^+$ and $w^-$ are independent of the lateral directions we see that
	\begin{align}
	\nabla w^+ \cdot \nu = \nabla w^- \cdot \nu =0 \text{ on } \partial D.
	\end{align}
	
	Since it suffices to check that $w^+$ is supersolution on $\{w^+ <1\}$ and $w^-$ is subsolution on $\{w^- >0\}$ we will restrict (without always mentioning it) us in the following to these sets. (Since $\max \{\cdot , \cdot \}$ of subsolutions is a subsolution and $\min \{\cdot, \cdot \}$ of supersolutions is a supersolution.)
	A calculation shows that
	\begin{align}
	\sL w^+ &= 
	\begin{cases}
	2(c_1+\dot{\xi}) \phi_1'(c_1+\xi(t)) -f(x,2\phi_1(c_1t + \xi(t)))    &  ,x_1 <0 \\
	\dot{\xi}(t) (\phi_1'(z_+)+\phi_1'(z_-)) + f_1(\phi_1(z_+)) + f_1(\phi_1(z_-)) 
	\\ -f(x,\phi_1(z_+)+\phi_1(z_-))   &  , x_1 >0
	\end{cases}\\
	&\geq
	\begin{cases}
	2(c_1+\dot{\xi}) \phi_1'(c_1+\xi(t)) -f_1(2\phi_1(c_1t + \xi(t)))     &,x_1 <0 \\
	\dot{\xi}(t) (\phi_1'(z_+)+\phi_1'(z_-)) + G(t,x_1)    &, x_1 >0
	\end{cases}
	\end{align}
	where $z_ +:= x_1 +c_1 t + \xi(t)$, $z_- := -x_1+c_1 t + \xi(t)$ and 
	\begin{align}
	G(t,x_1) = f_1(\phi_1(z_+)) + f_1(\phi_1(z_-)) -f_1(\phi_1(z_+)+\phi_1(z_-)) .
	\end{align}
	
	Using \eqref{Appendix:equ:xi} this can be rewritten as
	\begin{align}
	\sL w^+ \geq
	\begin{cases}
	2(c_1+M e^{\lambda (c_1 t + \xi(t))}) \phi_1'(c_1 t +\xi(t)) -f_1(2\phi_1(c_1t + \xi(t)))     &,x_1 <0 , \\
	M e^{\lambda (c_1 t + \xi(t))} (\phi_1'(z_+)+\phi_1'(z_-)) + G(t,x_1)    &, x_1 >0 .
	\end{cases}
	\end{align}
	
	Since $w^+$ is $C^2$ for $x_1 \neq 0$ and $C^1$ for all $x_1 \in \R$, therefore it suffices to check $\sL w^+ \geq 0$ on $x_1 >0$ and $x_1 <0$.
	
	On $x_1 <0$ we have that $\sL w^+ >0$ if we choose $T_1 \in (-\infty,T]$ sufficiently negative such that
	\begin{align} \label{Appendix:w+Supersolution_condition1}
	\phi_1(c_1 t + \xi(t)) \leq \frac{\theta_1}{2}  \text{ for } -\infty <t \leq T_1,
	\end{align}
	where $\theta_1$ is as in \eqref{F_4}.

	On $0 < x_1 \leq -(c_1+\xi(t))$ we can conclude from \eqref{Appendix:bound:phi1:zneg}, \eqref{Appendix:bound:phi1prime:zneg}, \eqref{Appendix:C11bound} that
	\begin{align}
	\sL w^+ &\geq M e^{\lambda (c_1 t + \xi)} (\phi_1'(z_+)+ \phi_1'(z_-)) - L\phi_1(z_+) \phi_1(z_-) \\
	&= M \gamma_0 e^{\lambda ( c_1 t + \xi)}  e^{\lambda (x_1+c_1 t+ \xi)} - L \beta_0^2 e^{\lambda (x_1 +c_1 t + \xi  )} e^{\lambda(-x_1 + c_1 t + \xi)} \\
	&= e^{2 \lambda (c_1 t + \xi)} (M \gamma_0 e^{\lambda x_1}- L \beta_0^2) .
	\end{align}
	Therefore if we choose $M >0$ such that
	\begin{align}\label{Appendix:w+Supersolution_condition2}
	M \gamma_0 > L \beta_0^2
	\end{align}
	we get $\sL w^+>0$ in this case.
	
	It  remains to show $\sL w^+ >0$ on $x_1 >-(c_1 t + \xi(t))$. Observe that
	
	\begin{align} \label{Appendix:Lw+Estim_second_case}
	\begin{split}
	\sL w^+ &\geq M e^{\lambda (c_1 t + \xi)} (\phi_1'(z_+) + \phi_1'(z_-)) - L \phi_1(z_+) \phi_1(z_-) \\
	&\geq M \gamma_1  e^{\lambda (c_1 t + \xi)} e^{- \mu ( x_1 + c_1 t + \xi)} - L \beta_0 e^{\lambda (-x_1 + c_1 t + \xi)} \\
	&\geq e^{\lambda (c_1 t + \xi)} ( M \gamma_1 e^{- \mu (x_1+c_1 t + \xi)} - L \beta_0 e^{- \lambda x_1}) .
	\end{split}
	\end{align}
	
	First in the subcase $\lambda \geq \mu$ we have $\sL w^+ >0$ provided that 
	\begin{align} \label{Appendix:w+Supersolution_condition3}
	M \gamma_1 > L \beta_0.
	\end{align}
	
	In the subcase $\lambda < \mu$ we have
	\begin{align}
	m_0 := -f_1'(0) < -f_1'(1) =: m_1. 
	\end{align}
	
	In this case it holds
	\begin{align} \label{Appendix:approx:estim_of_f_C11}
	f_1(u) +f_1(v) - f_1(u+v) = (m_1-m_0) v + O(v^2) + O(|v(1-u)|)
	\end{align}
	for $u \approx 1$ and $v \approx 0$.
	Hence $G(t, x_1)\geq 0$ if $z_+$ is very large and $z_-$ is very negative. This means that there is exists a constant $L_1 >0$ such that
	\begin{align}
	\sL w^+ \geq 0 \text{ if } x_1 \in [-(c_1 t + \xi(t)) + L_1 , \infty),
	\end{align}
	where we used that $c_1 t + \xi(t) \leq 0$.
	
	In the remaining subcase $x_1 \in [-(c_1 t + \xi(t)), -(c_1 t + \xi(t))+L_1 ]$, we see from \eqref{Appendix:Lw+Estim_second_case} that
	\begin{align}
	\sL w^+ &\geq e^{\lambda (c_1 t + \xi)} (M \gamma_1 e^{- \mu (x_1 + c_1 t + \xi)}) - L \beta_0 e^{-\lambda x_1 } ) \\
	&\geq e^{\lambda (c_1 t + \xi)} ( M \gamma_1 e^{-\mu L_1} - L \beta_0 e^{-\lambda x_1}   )
	\end{align}
	So in this subcase $\sL w^+>0 $ does hold if 
	\begin{align} \label{Appendix:w+Supersolution_condition4}
	M \gamma_1 e^{-\mu L_1} > L \beta_0.
	\end{align}
	Choosing $M>0$ and $T_1 \in (-\infty, T]$ such that \eqref{Appendix:w+Supersolution_condition1}, \eqref{Appendix:w+Supersolution_condition2}, \eqref{Appendix:w+Supersolution_condition3} and \eqref{Appendix:w+Supersolution_condition4} are fulfilled, we have constructed a supersolution $w^+$.

	Now we show that $w^-$ is a subsolution in the range ${w^- >0}$.

	A direct computation shows that it holds
	\begin{align}
	\sL w^- = 
	\begin{cases}
	0 &, x_1 <0 \\
	-M e^{\lambda (c_1 t + \xi)} (\phi_1'(y_+) - \phi_1'(y_-)) +H(t,x_1)         &, x_1 >0
	\end{cases} ,
	\end{align}
	where $y_+ := x_1 +c_1 t - \xi(t)$, $y_- := -x_1 + c_1 t - \xi(t)   $ and 
	\begin{align}
	H(t, x_1) = f_1(\phi_1(y_+)) - f_1(\phi_1(y_-)) - f_1(\phi_1(y_+)- \phi_1(y_-)) . 
	\end{align}
	Recall that on $x_1 >0$ it holds that $f(x,u) = f_1(u)$.

	Note that $w^- $ is $C^2$ except at $x_1 =0$ and $w^-$ has positive derivative gap at $x_1 =0$.  Therefore, in order to show that $w^-$ is a subsolution, it suffices to check that $\sL w^- \leq 0$ both for $x_1 <0$ and $x_1 >0$. 
	For $x_1 <0$ nothing needs to be checked. 
	
	In the range $0< x_1 \leq -(c_1- \xi(t))$, we note that 
	\begin{align}
	\phi_1'(y_+) - \phi_1'(y_-) = \int \limits_{y_-}^{y_+} \phi_1''(z) \dx{z} = \int \limits_{y_-}^{y_+} (c_1 \phi_1'(z)-f_1(\phi_1(z))) \dx{z} .
	\end{align}
	Since $y_- < y_+ <0$ in this range we have $\phi_1(z) < \theta_1$ for $z \in [y_-,y_+]$ by the normalization of $\phi_1$.
	By \eqref{F_4} it follows $f_1(\phi_1(z)) \leq 0$ in this range and hence
	\begin{align}
	\phi_1'(y_+) - \phi_1'(y_-) \geq c_1 (\phi_1(y_+)-\phi_1(y_-)) .
	\end{align}
	By \eqref{Appendix:C11bound} we have
	\begin{align}
	|H(t,x_1)| \leq L \phi_1(y_-) (\phi_1(y_+)-\phi_1(y_-))
	\end{align}
	Combining these, we obtain
	\begin{align}
	\sL w^- &\leq -c_1 M e^{\lambda (c_1 t + \xi)} (\phi_1(y_+) - \phi_1(y_-) ) + L \phi_1(y_-) (\phi_1(y_+) - \phi_1(y_-) ) \\
	&\leq (-c_1 M e^{\lambda (c_1 t + \xi)} + L \phi_1(y_-)) (\phi_1(y_+)- \phi_1(y_-)) \\
	&\leq (-c_1 M e^{\lambda (c_1 t + \xi) }  +L\beta_0 e^{\lambda(-x_1 + c_1 t - \xi(t) )}  ) (\phi_1(y_+) - \phi_1(y_-)) \\
	&\leq e^{\lambda c_1 t} (-c_1 M e^{\lambda \xi}+ L \beta_0 e^{\lambda (-x_1 - \xi(t))} ) ( \phi_1(y_+)- \phi_1(y_-)) .
	\end{align}
	Therefore $\sL w^- <0$ in this range provided that
	\begin{align} \label{Appendix:Lw-_bigger_zero_condition_1}
	c_1 M > L \beta_0 . 
	\end{align}
	Now we show that $\sL w^- <0$ in the range $x_1 > -(c_1-\xi(t))$ of $\{w^- >0\}$. Recall that $x_1 \geq 0\geq c_1 t -\xi(t)$.
	First we consider the subcase $\lambda \geq \mu$.
	It does then hold that 
	\begin{align}
	\sL w^- &\leq -M e^{\lambda (c_1 t + \xi)} (\phi_1'(y_+)-\phi_1'(y_-)) + L \phi_1(y_-) (\phi_1(y_+)- \phi_1(y_-)) \\
	&\leq  -M e^{\lambda (c_1 t + \xi)} (\gamma_1 e^{-\mu(x_1 +c_1 t -\xi)} -\delta_0 e^{\lambda (-x_1+c_1 t -\xi)}) + L \beta_0 e^{\lambda (-x_1 +c_1 t -\xi)} \\
	&= -M e^{\lambda(-x_1 + c_1 t + \xi)} (\gamma_1 e^{-\mu ( c_1 t - \xi )+(\lambda - \mu)x_1} - \delta_0 e^{\lambda (c_1 t -\xi)}- M^{-1}L \beta_0 e^{-2\lambda \xi}) \label{Appendix:w-_lambda_bigger_mu_estimate_on_L} \\
	&\leq M e^{\lambda (-x_1 +c_1 t + \xi )} (\gamma_1 e^{-\mu (c_1 t - \xi)}-\delta_0 e^{\lambda (c_1 t - \xi)}- M^{-1} L \beta_0) .
	\end{align}
	Thus choosing $T_1 \in (-\infty,T]$ sufficiently negative such that
	\begin{align} \label{Appendix:Lw-_bigger_zero_condition_2}
	\gamma_1 e^{-\mu ( c_1 t -\xi)} - \delta_0 e^{\lambda (c_1 t -\xi)} -M^{-1} L \beta_0 >0
	\end{align}
	we have that $\sL w^- <0$.
	
	In the case $\lambda < \mu$ as in \eqref{Appendix:approx:estim_of_f_C11} we find that
	\begin{align}
	H(t,x_1) = -(m_1-m_0) \phi_1(y_-) + O(\phi_1^2(y_-)) +O(\phi_1(y_-)(1-\phi_1(y_+))) \leq -\tilde{C} \phi_1(y_-)
	\end{align}
	for some constant $\tilde{C}>0$ since in the case $\lambda < \mu$ we have that $m_1 >m_0$, in the realm where $y_+$ is big enough and positive and $y_-$ is sufficiently negative.
	Consequently there exists $L_2 >0$ such that 
	\begin{align}
	H(t,x_1) \leq - \tilde{C} \beta_0  e^{\lambda (-x_1 +c_1 t - \xi)} \text{ if } x_1 \in [-(c_1 t -\xi(t))+L_2, \infty).
	\end{align}
	Therefore 
	\begin{align}
	\sL w^- \leq M e^{\lambda (c_1 t + \xi)} \phi_1'(y_-) + H(t,x_1) \leq e^{\lambda (-x_1 + c_1 t -\xi)} (M \delta_0 e^{\lambda (c_1 t +\xi)}-\tilde{C}\beta_0 ).
	\end{align}
	It follows that again $\sL w^-<0$ given that $T_1 \in (-\infty,T]$ is chosen sufficiently negative so that 
	\begin{align} \label{Appendix:Lw-_bigger_zero_condition_3}
	M \delta_0 e^{\lambda (c_1 t + \xi)} < \tilde{C} \beta_0 \text{ for } -\infty < t \leq T_1.
	\end{align}
	Finally in the range $x_1 \in [-(c_1 t - \xi(t)), -(c_1 t - \xi(t))+ L_2]$, we see from \eqref{Appendix:w-_lambda_bigger_mu_estimate_on_L} that
	\begin{align}
	\sL w^- \leq - M e^{\lambda (-x_1 + c_1 t + \xi)} (\gamma_1 e^{-\lambda (c_1 t - \xi)- (\mu - \lambda ) L_2}  - \delta_0 e^{\lambda (c_1 t - \xi)} -M^{-1} L \beta_0 e^{-2 \lambda \xi} )
	\end{align}
	Again we have $\sL w^- <0$ in this range provided that $T_1 \in (-\infty, T]$ is chosen sufficiently negative such that
	\begin{align} \label{Appendix:Lw-_bigger_zero_condition_4}
	\gamma_1 e^{-\lambda (c_1 t - \xi ) - (\mu -\lambda) L_2} - \delta_0  e^{\lambda (c_1 t - \xi)} -M^{-1} L \beta_0 >0 \text{ for } - \infty <t \leq T_1 .
	\end{align}
	Combining these we see that $w^-$ is a subsolution of \eqref{DiffEqu}, if $M>0$ and $T_1 \in (-\infty, T]$ are chosen so that \eqref{Appendix:Lw-_bigger_zero_condition_1}, 	\eqref{Appendix:Lw-_bigger_zero_condition_2}, \eqref{Appendix:Lw-_bigger_zero_condition_3}, \eqref{Appendix:Lw-_bigger_zero_condition_4} hold.

	In following we will use the sub-and supersolutions $w^-$ and $w^+$ to construct the entire solution.  
	In order to also ensure that the constructed entire solution will also be strictly increasing in time, we replace the argument given in \cite{MatanoObst} by an argument kindly pointed out to the author by François Hamel. 
	
	Let us therefore define the non-decreasing (in time) modification 
	\begin{align}
		\tw^-(t,x) := \sup \limits_{s<0} w^-(t+s,x) ~,
	\end{align}
	of the subsolution $w^-$ that still is a subsolution (in the viscosity sense).
	
		Let us now construct a sequence of solutions $u_n$ of \eqref{DiffEqu} defined for $-n \leq t < \infty$ with initial condition
		\begin{align}
		u_n(-n,x) = w^- (-n,x).
		\end{align}

		Observe that by construction $w^- \leq w^+$ and $\partial_t w^+ <0$ for $t$ sufficiently negative. Since $\tw^-$ is a subsolution we get
		\begin{align}
	u_n(-n,x) = \tw^-(-n,x) \leq \sup \limits_{s<0} w^+(-n+s,x) \leq w^+(-n,x) 
	\end{align}
	for $n$ large enough and all $x \in D$. Let in the following always $n$ be large enough.
	Since $\tw^-$ is a subsolution and $w^+$ is a supersolution we get that
	\begin{align}
		\tw^-(t,x) \leq u_n(t,x) \leq w^+(t,x) \text{ for all } t \in [-n,T_1] .
	\end{align}
	From this follows directly (setting $t=-(n-1)$) that $u_n$ is non-decreasing in $n$:
	\begin{align}
		u_n(-n+1,x) \geq \tw^-(-n+1,x) = u_{n-1}(-n+1,x) \text{ for all } x \in D.
	\end{align}
	Then by the comparison principle again
	\begin{align}
		u_n(t,x) \geq u_{n-1} (t,x) \text{ for all } t \in [-n+1,T_1] \text{ and } x \in D.
	\end{align}

	Letting $n \rightarrow \infty$ and using parabolic estimates we see that this sequence converges (up to a subsequence) to an entire solution ${u}$ of \eqref{DiffEqu} satisfying
	\begin{align}
		\tilde{w}^-(t,x) \leq {u}(t,x) \leq w^+(t,x) \quad  \text{ for all } (t,x) \in (-\infty, T_1]\times D.
	\end{align}
	
	Let us now show that $\partial_t {u} >0$.
	By construction $\tw^-$ is non-decreasing in time  and hence for all $h>0$ such that $-n+h<T_1$
	\begin{align}
	u_n(-n+h,x) \geq \tw^-(-n+h,x) \geq \tw^-(-n,x) =u_n(-n,x) .
	\end{align}
	This implies directly that
	\begin{align}
		\partial_t u_n(-n,x) \geq 0 \text{ for all } x \in D
	\end{align}
	and hence by the maximum principle
	\begin{align}
		\partial_t u_n(t,x) \geq 0 \text{ for all } t\in[-n,\infty) \text{ and } x \in D.
	\end{align}
	
	Passing to the limit $n \rightarrow \infty$ it follows that 
	\begin{align} 
		\partial_t u(t,x) \geq 0 \text{ for all } t \in \R .
	\end{align}
	But since $u$ being a solution of \eqref{DiffEqu} with initial condition \eqref{Anfangsbedingung} cannot be constant, from the strong maximum principle it follows that
	\begin{align} \label{Appendix:time_derivative_of_solution_is_positive}
		\partial_t u(t,x) > 0 \text{ for all } t \in \R.
	\end{align}

	We are now in the position to turn to uniqueness. Let us start introducing the notion of a transition zone first.
	Given $\eta \in (0,\frac{1}{2} ]$ we define for each $t \in \R$
	\begin{align}
		D_\eta(t) := \set {  x \in D  | \eta \leq {u}(t,x) \leq 1-\eta} .
	\end{align}
	$D_\eta(t)$ can be understood as transition zone of the front $u$ at time $t$. By condition \eqref{Anfangsbedingung} for any $\eta \in (0, \frac{1}{2}]$ we can find $T_\eta \in \R$ and $M_\eta \geq 0$ such that
	\begin{align} \label{Appendix:condition_on_transition_zone}
		D_\eta(t) \subset \set { x \in D | \abs {c_1t +x_1} \leq M_\eta    } \subset \set { x \in D | x_1 \geq 1  }
	\end{align}
	for $-\infty < t \leq T_\eta$. 
	The argument will rely on the following 
	
	\begin{Lemma} \label{Appendix:Lemma:bound_from_below_on_temporal_derivative_on_transition_zone}
		For any $\eta \in (0, \frac{1}{2}]$ there exists $\delta >0$ such that
		\begin{align} \label{Appendix:bound_on_time_derivative_in_transition_zone}
			\partial_t u \geq \delta \quad \text{ for } t \in (-\infty, T_\eta) , x \in D_\eta(t) .
		\end{align} 
	\end{Lemma}

\begin{proof}[Proof of the Lemma]
	Suppose that \eqref{Appendix:bound_on_time_derivative_in_transition_zone} does not hold. Then there exists a sequence $(t_k)_{k \in \N} \subset (-\infty, T_\eta]$ and $(x_k)_{k \in \N} = ((x_k)_1, \dots, (x_k)_N)_{k \in \N} \subset D_\eta(t)$ such that
	\begin{align}
		\partial_t u (t_k,x_k) \rightarrow 0 \text{ as } k \rightarrow \infty.
	\end{align} 
	Without loss of generality we can assume that either $(t_k)_{k \in \N}$ converges to some $t^* \in (-\infty, T_\eta]$ or $t_k \rightarrow -\infty$ as $k \rightarrow \infty$.
	In the former case  \eqref{Appendix:condition_on_transition_zone} implies that $((x_k)_1)_{k \in \N}$ is bounded. So we can assume that $(x_k)_1 \rightarrow x^*_1$ as $k \rightarrow \infty$. Let us now set
	\begin{align}
		u_k(t,x) := u(t,x+x_k) \text{ for all } (t,x) \in (-\infty, T_\eta] \times\set { x \in D | x_1 \geq 1  } .
	\end{align}
	Then by parabolic estimates it follows that 
	\begin{align}
		u_k(t,x) \rightarrow u^*(t,x) \in C^{1,2}_\text{loc}((-\infty, T_\eta] \times \set { x \in D | x_1 \geq 1  }).
	\end{align}
	The limit function $u^*$ does satisfy \eqref{DiffEqu} on $(-\infty, T_\eta] \times \set { x \in D | x_1 \geq 1  }$ and by \eqref{Appendix:time_derivative_of_solution_is_positive} and the convergence above we find that
	\begin{align}
		\partial_t u^*(t^*, 0 ) = 0 \text{ and } \partial_t u^*(t,x) \geq 0 \text{ in } (-\infty, T_\eta] \times \set { x \in D | x_1 \geq 1  } .
	\end{align}
	Applying the strong maximum principle to $\partial_t u^*$ we obtain
	\begin{align}
		\partial_t u^* \equiv 0 \text{ in } (-\infty, T_\eta] \times \set { x \in D | x_1 \geq 1  } .
	\end{align}
	But this is impossible since \eqref{Anfangsbedingung} implies that
	\begin{align}
		u^*(t,x) - \phi_1(x_1 + x^*_1+ c_1 t) \rightarrow 0 \text{ as } t \rightarrow - \infty \text{ uniformly in }\set { x \in D | x_1 \geq 1  } .
	\end{align}
	
	In the case where $t_k \rightarrow - \infty$ as $k \rightarrow \infty$ we set
	\begin{align}
		u_k(t,x) := u(t+t_k,x+x_k) \text{ in } (-\infty, T_\eta] \times \set { x \in D | x_1 \geq 1  } .
	\end{align}
	Then as above we can find a subsequence such that $u_k \rightarrow u^*$ in $ C^{1,2}_\text{loc}((-\infty, T_\eta] \times \set { x \in D | x_1 \geq 1  })$, where
	\begin{align}
	\partial_t u^* (0,0) =0 \text{ and } \partial_t u^* (t,x) \geq 0 \text{ in } (-\infty, T_\eta] \times \set { x \in D | x_1 \geq 1  } 
	\end{align}
	and again by the strong maximum principle we conclude
	\begin{align}
		\partial_t u^* \equiv 0 \text{ in } (-\infty, T_\eta] \times \set { x \in D | x_1 \geq 1  } .
	\end{align}
	But this is again impossible because up to choosing a subsequence it holds that
	\begin{align}
		u^*(t,x) = \phi_1(x_1+c_1t + \alpha) , \text{ where } \alpha \in [-M_\eta, M_\eta] . 
	\end{align}
	Hence the Lemma is proved.
\end{proof}

With the help of this Lemma we can now prove the uniqueness of entire solutions of \eqref{DiffEqu} subject to \eqref{Anfangsbedingung}. Suppose there exists another entire solution $v$ of \eqref{DiffEqu} satisfying \eqref{Anfangsbedingung}. We choose $\eta >0$ sufficiently small such that 
\begin{align} \label{Appendix:bound_on_du_f_near_0_and_1}
	\partial_u f(x,s) \leq - \beta \text{ for } s \in [0, 2 \eta] \cup [1-2 \eta, 1] , x \in D
\end{align}
for some $\beta >0$. 
Then for any $\epsilon \in (0, \eta)$ we find $t_\epsilon \in \R$ such that 
\begin{align} \label{Appendix:choice_of_t_epsilon_from_AB}
	\norm {v(t, \cdot) - u(t, \cdot)}_{L^\infty} < \epsilon \text{ for } - \infty < t < t_\epsilon .
\end{align}
Let us define the candidates for our super- and subsolutions for each $t_0 \in (-\infty, T_\eta - \sigma \epsilon]$ as 

\begin{align}
	W^+ (t,x) &:= \min \set {1, u(t_0+t+\sigma \epsilon (1-e^{-\beta t}), x )  + \epsilon e^{-\beta t } } , \\
	W^- (t,x) &:= \max \set {0, u(t_0+t-\sigma \epsilon (1-e^{-\beta t}), x )  - \epsilon e^{-\beta t } } ,
\end{align}
	where $\sigma>0$ will be specified later. Now by \eqref{Appendix:choice_of_t_epsilon_from_AB}
	\begin{align} \label{Appendix:initial_condition_comparison_argument_uniqueness_proof}
		W^-(0,x) \leq v(t_0,x) \leq W^+(0,x) \text{ for }  x \in D.  
	\end{align}
	Next we show that $W^+$ and $W^-$ are super- and subsolutions in the time range $t \in [0, T_\eta - t_0 - \sigma \epsilon ]$. It suffices to check this on $\set {W^+ <1}$, since $1$ already is a solution of \eqref{DiffEqu}. We see that
	\begin{align}
		\sL W^+ &= \partial_t W^+ - \Delta W^+ - f(x, W^+)  \\
		&= \sigma \epsilon \beta e^{-\beta t } \partial_t u - \epsilon \beta e^{-\beta t} + f(x,u) - f(x,u+\epsilon e^{-\beta t}) \\
		&= \epsilon e^{-\beta t} \bra { \sigma \beta \partial_t u - \beta - \partial_u f (x, u + \theta e^{-\beta t})      } ,
	\end{align}
	where $\theta(t,x)$ is some function such that $\theta(t,x) \in [0,1]$. (For $u$ and $\partial_t u$ we have always omitted the arguments $( t_0 + t + \sigma \epsilon (1-e^{-\beta t}) ,x   )$ .)
	
	We do now distinguish the following cases:
	
	If $x \in D_\eta(t+t_0 + \sigma \epsilon (1-e^{-\beta t}) )    $, then by Lemma \ref{Appendix:Lemma:bound_from_below_on_temporal_derivative_on_transition_zone} 
	
	\begin{align}
		\sL W^+ \geq \epsilon e^{- \beta t} \bra {  \sigma \beta \delta - \beta - \norm { \partial_u f  }_{L^\infty} }
	\end{align}
   Therefore we get $\sL W^+>0$ if we choose $\sigma >0$ sufficiently large (independently of $\epsilon$). 
   
   On the other hand, if $x \notin D_\eta(t+t_0 + \sigma \epsilon (1-e^{-\beta t})) $ then 
   
   \begin{align}
   	u + \theta \epsilon e^{-\beta t} \in [0,2\eta] \cup [1-\eta, 1]
   \end{align}
   and consequently by \eqref{Appendix:bound_on_du_f_near_0_and_1} one sees that $\partial_u f (x, u+\theta \epsilon e^{-\beta t}) \leq - \beta$ and thereby
   \begin{align}
   \sL W^+ \geq \epsilon e^{-\beta t} (-\beta + \beta ) =0 .
   \end{align}

   Hence we know that $\sL W^+ \geq 0$ for all $t \in [0, T_\eta - t_0 - \sigma \epsilon], x \in D$. Similarly one proves $\sL W^- \leq 0$ in this region.
   Together with \eqref{Appendix:initial_condition_comparison_argument_uniqueness_proof} we get that
   \begin{align}
   	W^-(t,x) \leq v(t_0+t,x) \leq W^+(t,x) \text{ for } t \in [0, T_\eta -t_0 -\sigma \epsilon] , x\in D.
   \end{align}
	Substituting $t_0 + t$ by $t$ we can rewrite this inequality as
	\begin{align}
		u(t-\sigma \epsilon (1-e^{-\beta (t-t_0)}),x) - \epsilon e^{-\beta ( t- t_0)} 
		&\leq v(t,x ) \\
		&\leq u(t+\sigma \epsilon (1-e^{-\beta (t-t_0)}),x) + \epsilon e^{-\beta (t-t_0)}
	\end{align}
	for $t \in [t_0, T_\eta - \sigma \epsilon]$ and (still) $t_0 \in (-\infty , T_\eta - \sigma \epsilon]$. Letting $t_0 \rightarrow - \infty$, we obtain 
	\begin{align}
		u(t-\sigma \epsilon , x ) \leq v(t,x) \leq u(t+\sigma \epsilon , x) \text{ for all } t \in (-\infty, T_\eta -\sigma \epsilon], x \in D. 
	\end{align}
	Hence by the comparison principle the above inequalities hold for all $(t,x) \in \R \times D$. Letting $\epsilon \rightarrow 0$ ($\sigma$ was independent of $\epsilon$) we find
	\begin{align}
		u \equiv v \text{ in } \R \times D
	\end{align}
	and the proof is finished.

\end{proof}

%
%


\nomenclature{$\mathcal{L}^N$}{The $N$-dimensional Lebesgue measure}%
\nomenclature{$B_R(x)$}{The open ball of radius $R$ centred in $x$}%
\nomenclature{$C_R(x)$}{The open cube of side length $R$ centred in $x$}%
\nomenclature{$\omega_N$}{The $N$-dimensional Lebesgue-measure of the $N$-dimensional unit ball}%

%
%

\bibliographystyle{abbrv}
\bibliography{Change_of_speed_v3.bib}

\end{document}